\documentclass{amsart}
\usepackage{amsfonts}
\usepackage{pgfplots}
\usepackage{amsmath}
\usepackage{calc}
\usepackage{amsthm}
\usepackage{amssymb}
\usepackage{hyperref}
\usepackage{autonum}
\usepackage[utf8]{inputenc}
\usepackage{relsize}
\usepackage[T1]{fontenc}
\usepackage{soul}

\usepackage{tikz}
\usepackage{tikz-cd}

\pgfplotsset{compat=newest}

\newcommand{\vlad}{\color{brown}}

\newtheorem{theorem}{Theorem}
\newtheorem*{theorem*}{Theorem}
\newtheorem{corollary}[theorem]{Corollary}
\newtheorem{proposition}[theorem]{Proposition}
\newtheorem{lemma}[theorem]{Lemma}
\newtheorem{definition}[theorem]{Definition}

\newtheorem{hypothesis}[theorem]{Hypothesis}

\newtheorem{remark}[theorem]{Remark}

\newcommand{\calO}{\mathcal{O}}
\newcommand{\OK}{\mathcal{O}_K}

\newcommand{\KR}{K_\mathbb{R}}
\newcommand{\R}{\mathbb{R}}
\newcommand{\Q}{\mathbb{Q}}

\DeclareMathOperator{\SL}{SL}
\DeclareMathOperator{\ML}{Mod}
\DeclareMathOperator{\GL}{GL}

\DeclareMathOperator{\cl}{cl}
\DeclareMathOperator{\Cl}{cl}

\DeclareMathOperator{\argmin}{argmin}
\DeclareMathOperator{\vol}{vol}

\DeclareMathOperator{\card}{\texttt{\#}}

\DeclareMathOperator{\rank}{rk}
\DeclareMathOperator{\spant}{span}
\DeclareMathOperator{\ind}{\mathlarger{\mathbf{1}}}

\DeclareMathOperator{\N}{N}
\DeclareMathOperator{\Tr}{tr}

\DeclareMathOperator{\Gr}{\mathbf{Gr}}

\newcommand{\nihar}{\color{red}}

\begin{document}
\title{Effective module lattices and their shortest vectors}
\author[N. Gargava]{Nihar Gargava}
\author[V. Serban]{Vlad Serban}
\author[M. Viazovska]{Maryna Viazovska}
\author[I. Viglino]{Ilaria Viglino}

\address{N. Gargava, Section of Mathematics, EPFL, Switzerland}
\email{nihar.gargava@epfl.ch}
\address{V. Serban, Department of Mathematics, New College of Florida, U.S.A.}
\email{vserban@ncf.edu}
\address{M. Viazovska, Section of Mathematics, EPFL, Switzerland}
\email{maryna.viazovska@epfl.ch}
\address{I. Viglino, Section of Mathematics, EPFL, Switzerland}
\email{ilaria.viglino@epfl.ch}

\maketitle

\begin{abstract}
We prove tight probabilistic bounds for the shortest vectors in module lattices over number fields using the results of \cite{GSV23}. Moreover, establishing asymptotic formulae for counts of fixed rank matrices with algebraic integer entries and bounded Euclidean length, we prove an approximate Rogers integral formula for discrete sets of module lattices obtained from lifts of algebraic codes. This in turn implies that the moment estimates of \cite{GSV23} as well as the aforementioned bounds on the shortest vector also carry through for large enough discrete sets of module lattices. 
\end{abstract}

\section{Introduction}
  \label{se:intro}

One of the central concerns in lattice-based cryptography is to prove or disprove the existence
of fast algorithms that can find, even if approximately, the length of the shortest vectors of a lattice.
The non-existence of such algorithms is still a wide-open question, and there exist well-publicized open challenges \cite{BLR08}
inviting competitors to submit algorithms solving the shortest vector problem (SVP) for "randomly" produced lattices. For random lattices in increasing dimension $d$, it turns out that the length of the shortest vector can be predicted almost exactly:
\begin{theorem}
\cite[Theorem 5]{LN20}
\label{th:phong}
Let $\Lambda\subseteq \mathbb{R}^{d}$ be a Haar-random lattice chosen from the space $\SL_{d}(\mathbb{R})/\SL_{d}(\mathbb{Z})$. Then, as $d \rightarrow \infty$, with probability $1- o(1)$ we have 
\begin{equation}
  1-\frac{\log \log d}{d} \leq \frac{\lambda_1(\Lambda)}{\gamma(d)} \leq 1 + \frac{\log\log d}{d},
\end{equation}
where $\gamma(d)$ is radius of a ball of unit volume in $d$ dimensions. 
\end{theorem}
\begin{remark}
\label{re:gamma}
For large $d$, $\gamma(d) \simeq \sqrt{\tfrac{d}{2 \pi e}}$ by Stirling's approximation.
\end{remark}

A weaker form of this and similar results is stated by Ajtai in \cite{A02} without proof.
In the lattice challenges in \cite{BLR08}, this result is used to set benchmarks for challenge contenders. Such results greatly inform our understanding and analysis of lattice reduction algorithms \cite{LN20,EK20}. \par

One of the important questions in current lattice-based cryptography concerns the analysis of the hardness of problems such as SVP and related on lattices with additional algebraic structure that appear in efficient cryptographic implementations. Module lattices are algebraically $\OK$-modules, where $\OK$ is the ring of integers of a number field $K$, and the latter is often taken to be a cyclotomic field. They are one of the main special classes of lattices studied, see e.g. \cite{LS15,DMPT23}. See also work on ring-based versions of the learning with errors (LWE) and short integer solutions (SIS) problems such as \cite{RingLWE2010,RingLWE}.
Producing asymptotic estimates like Theorem \ref{th:phong} for such lattices in growing dimensions is a bit trickier because in this setting one often desires to understand the behavior when increasing the degree of the number field.

In this article we obtain analogous results to the above theorem for many families of module lattices. 
For example, we obtain:
\begin{theorem}
	\label{th:as_module_bound}
Let $t \geq 27$ be fixed.
Let $K = \mathbb{Q}(\mu_{k})$ be a cyclotomic number field and let $d = t \cdot \deg K = t \varphi(k)$. 
Consider a Haar-random 
unit covolume
module lattice $\Lambda \subseteq K^{t} \otimes \mathbb{R}$ in the moduli space of rank-$t$ module lattices over $K$ (see Section \ref{se:module}).
Then, as $k \rightarrow \infty$, we have with probability $1- o(1)$ that
\begin{equation}
 1- \frac{\log \log k}{d} \leq k^{-\frac{1}{d}}\frac{\lambda_1(\Lambda)}{\gamma(d)} \leq  1 + \frac{\log \log k}{d}.
 \label{eq:bound}
\end{equation}
\end{theorem}
The $t \geq 27$ condition is likely not tight and appears for technical reasons. Improvements to our methods might further reduce this lower bound to values like 3,4,5 (we need $t > 2$ for two moments to exist). Note that in the statement above $k^\frac{1}{d}  \simeq 1 + \tfrac{\log k}{t\varphi(k)}.$
The choice of the sequence of number field can be changed as long as one has uniform bounds on the absolute Weil height of algebraic numbers. This is always true for abelian extensions such as cyclotomic fields. See Section \ref{se:bounds_for_vectors} for definitions and more details.

Our main tool in studying these random module lattices is the well-known Rogers integration formula \cite{K19,hughes2023mean,GSV23}, which then leads us to moment estimates for the number of non-trivial lattice points of bounded Euclidean length as a random variable over the probability space of Haar-random lattices. \par
This study was carried through by the first three authors in \cite{GSV23}. However, in this work we refine these results and show that the same integration formula and moment estimates also hold for effective lattice constructions in the sense of \cite{GM03,M16} obtained from pre-images of algebraic codes via projection maps from characteristic zero. Since averages over these ``lifts of codes'' constructions nicely 
approximate averages with respect to the Haar-induced measure on module lattices,
results such as Theorem \ref{th:as_module_bound} can also be attained for these discretized sets of lattices coming from algebraic codes of increasing size. This should be interesting from a computational and coding-theoretic perspective.

While establishing an effective Rogers integral formula (see Theorem \ref{th:higher_moments} for the statement), one of the hurdles we encounter is needing a generalization of the following result of Katznelson \cite{K1994} which originally appeared for $K=\mathbb{Q}$. We believe this result to be of independent interest.
\begin{theorem}\label{thm:KatznelsonOK}
For $m\leq n<t$, let $N(T;m,n,t,K)$ denote the set of $n \times t$ matrices with coefficients in $\OK$ and rank $m$ such that each entry has norm bounded by $T$ with respect to a suitable norm on $\KR$. Then, for some explicit constant $C>0$, one has for any $\varepsilon > 0$ that as $T$ tends to infinity
\begin{equation}
    N(T;m,n,t,K) = C \cdot T^{mt \deg K} + O(T^{mt\deg K-1+\varepsilon}).
\end{equation}

\end{theorem}
This result follows from Proposition \ref{pr:convergence} by substituting $g$ as the indicator function of appropriate balls in the proposition. Similarly, the value of the constant $C>0$ can be derived from our results.

\section*{Acknowledgements}

We would like to thank Andreas Str\"ombergsson for comments on our preprint \cite{GSV23} that motivated the writing of much of this article. The writing has also benefited from some discussions with Phong Nguyen, Thomas Espitau and Seungki Kim.

This research was partly funded by the Swiss National Science Foundation (SNSF), Project funding (Div. I-III), "Optimal configurations in multidimensional spaces",
184927. 

\section{Moduli space of module lattices}
\label{se:module}

Let $K$ be a number field of signature $(r_1,r_2)$, $r_1 + 2 r_2 = \deg K$,
and let $\OK$ be the ring of integers of the number field. Let $t$ be a positive integer.
We fix, for once and for all, the following  inner product on $K \otimes \mathbb{R} \simeq \mathbb{R}^{r_1} \times \mathbb{C}^{r_2}$
\begin{equation}
  \langle x,y\rangle = |\Delta_K|^{-\frac{2}{\deg K}}  \Tr(x\overline{y}),
  \label{eq:norm}
\end{equation}
where $\Delta_K$ is the discriminant of the number field. The involution $\overline{(\ ) }$ in \eqref{eq:norm}
denotes complex conjugation
on all the complex places of $K$.
Observe that, equipped with such a quadratic form, the lattice $\OK \subseteq K \otimes \mathbb{R}$ has unit covolume. For any $n \in \mathbb{Z}_{\geq 1}$, the Euclidean space $K^{n} \otimes \mathbb{R} = (K\otimes \mathbb{R})^{n}$ comes equipped with the structure from $n$-fold copies of this underlying inner product.
For notational simplicity, we will abbreviate $\KR = K \otimes \mathbb{R}$. 

The goal of this research is to consider random rank-$t$ $\OK$-modules inside $\KR^t$. 
We start with the setup of \cite{DK22} and the following definition.
\begin{definition}
A module lattice in $\KR^{t}$ is a pair 
$(g,M)$ where $g \in \GL_{t}(K \otimes \mathbb{R})$ and $M \subseteq K^{t}$ is a finitely generated $\OK$-module of maximal rank so that 
\begin{equation}
	\vol\left(K^{t} \otimes \mathbb{R} / (g^{-1} \cdot M ) \right) = 1.
\end{equation}
Two module lattices are considered equal if the lattice $g^{-1}M \subseteq K^{t} \otimes \mathbb{R}$ is identical.

A module lattice is called rational, or simply an $\OK$-module, if it is equal to $(1,M)$ for some $M \subseteq K^{n}$.
When not ambiguous, we may refer to $M$ as the module lattice $(1, M)$. The set of all module lattices of rank $t$ over $K$ shall be denoted by $\ML_{t}(K)$.
\end{definition}

\begin{remark}
We only consider unit determinant lattices in this definition. In the literature, module lattices don't necessarily have unit determinant. However, dropping the determinant condition leaves a space of lattices that has infinite volume with respect to the Haar measure and therefore is not a probability space. Moreover, the questions addressed in this work really only concern lattices up to scaling. 
\end{remark}

On the space of module lattices, we have a natural left action of $\GL_{t}(\KR)$ as 
\begin{equation}
 h :(g,M)  \mapsto (  |\N(\det h)|^{-\frac{1}{t \deg K}}h g , M).
\end{equation}
This descends to an action of $\GL_t(\KR)/\mathbb{R}_{>0}$.

\subsection{Steinitz classes}

For each module lattice, one can define its Steinitz class:
\begin{definition}
Let $\Lambda = (g,M)$ be a module lattice with $\det(\Lambda)=1$. 
Let $\cl(K)$ be the class group of $K$.
Then the Steinitz class of $(g,M)$ is denoted by 
\begin{equation}
[\Lambda] =  [\mathcal{I}] \in \cl(K),
\end{equation}
where $\mathcal{I}$ is a fractional ideal in $K$ 
generated by the set
\begin{equation}
  \{ \det( v_1,v_2,\dots,v_t) \mid v_1,v_2,\dots,v_t \in M\}.
\end{equation}
That is, $[\Lambda]$ is the ideal class of the fractional ideal generated by all the determinants of $t$-tuples of vectors in $M$.
\label{de:steinitz}
\end{definition}

\begin{remark}
Observe that the Steinitz class is invariant under the action of $\GL_{t}(K)$. 
Indeed, such an action will only change the determinants generating the fractional ideal $\mathcal{I}$ 
in Definition \ref{de:steinitz}
by a uniform scalar in $K^{\times}$. 

One therefore sees that Definition \ref{de:steinitz} is well-posed in the sense that if $(g_1,M_1) = (g_2,M_2)$ are the same module lattice, 
then $g_1g_2^{-1}$ is an $\OK$-linear map from $M_2$ to $M_1$ and hence
\begin{equation}
  g_1g_2^{-1} \in \GL_{t}(K) \subsetneq \GL_{t}(K\otimes \mathbb{R}).
\end{equation}
\end{remark}

One should view the definition of Steinitz class in light of the classification of torsion-free modules over a Dedekind domain. Applied to $\OK$ this yields:
\begin{proposition}
As an abstract $\OK$-module, a rank $t$ torsion-free $\OK$ module is isomorphic to 
    $$\OK^{t-1}\oplus \mathfrak{a}$$
for some $\OK$-ideal $\mathfrak{a}$. Moreover, two such modules $\OK^{t-1}\oplus \mathfrak{a}_1$ and $\OK^{t-1}\oplus \mathfrak{a}_2$ are isomorphic if and only if there exists $x\in K^\times$ with $\mathfrak{a}_1=x\cdot \mathfrak{a}_2$. 
\end{proposition}
The Steinitz class of $\OK^{t-1}\oplus \mathfrak{a}$ is then simply the ideal class $[\mathfrak{a}]$ of $\mathfrak{a}$ and classifies the lattices as abstract $\OK$-modules.  
We then have a transitive action of the Lie group $\GL_t(\KR)/\mathbb{R}_{>0}$ on module lattices of fixed Steinitz class. This implies that each connected component of $\ML_t(K)$ looks like 
$\GL_t(\KR)/\Gamma \cdot \mathbb{R}_{>0}$ for some arithmetic group $\Gamma$.
One can write this 
as the following ``exact'' sequence of sets. 
\begin{equation}
\GL_{t}(\KR)/ (\mathbb{R}_{>0} \cdot \GL_{t}(\OK)) \hookrightarrow  \ML_t(K) \twoheadrightarrow \cl(K) 
\end{equation}
From this, we can then extend the structure to the level of topological measure spaces. This produces a natural probability measure on $\ML_{t}(K)$, each of whose components have the Haar measure as a homogeneous space of $\GL_t(\KR)/\mathbb{R}_{>0}$. 
Note that for $t\geq 2$, $\ML_{t}(K)$ is a non-compact probability space.

\begin{remark}
\label{re:adelic}

One could take an adelic point of view as in \cite{DK22} and see $\ML_t(K)$ as an adelic quotient. 
In this way, $\ML_t(K)$ really is seen to carry a Haar measure as a homogeneous space of a locally compact group.
\end{remark}

\subsection{Lifts of codes and Haar measure}

In practice, we will consider module lattices that can be generated from algebraic codes as follows: consider unramified prime ideals $\mathcal{P}\subseteq \mathcal{O}_K$ and fix a number $s \in \{1,\dots,d-1\}$. If $\pi_\mathcal{P}:\mathcal{O}_K \rightarrow k_\mathcal{P}$ is the projection map to the residue field $k_\mathcal{P}$, we consider the set of unit covolume lattices:

\begin{equation}
\mathcal{L}(\mathcal{P},s) = \{ \beta \pi_\mathcal{P}^{-1}(S) \mid S\subseteq k_\mathcal{P}^{t} \text{ is an $s$-dimensional $k_\mathcal{P}$-subspace}\}
\label{eq:def_of_L}
\end{equation}
where 
\footnote{
	There is an abuse of notation in Equation (\ref{eq:def_of_L}) as $\pi_\mathcal{P}$ denotes a map on $\mathcal{O}_K^{t}\rightarrow k_\mathcal{P}^{t}$. Henceforth, this map is to be understood as ``applying the mod $\mathcal{P}$ operation at each coordinate''.
}
\begin{equation}
\beta = \beta(\mathcal{P},s) = \N(\mathcal{P})^{-\left(1-\frac{s}{t}\right)\frac{1}{[K:\mathbb{Q}]}}.
\end{equation}
Indeed, this scaling factor $\beta$ ensures that the lattice $ \beta \pi_\mathcal{P}^{-1}(S)$ has the same covolume as $\mathcal{O}_K^{t} \subseteq K_\mathbb{R}^{t}$.

The lattices $\mathcal{L}(\mathcal{P},s) \subset \ML_{t}(K)$ are module lattices. They have the same Steinitz class $[\mathcal{P}]^{t-s}\in \Cl(K)$ for fixed $\mathcal{P},s$ and therefore lie on the same connected component of $\ML_{t}(K)$. The ideal classes $[\mathcal{P}]$ are known to equidistribute in the class group as the prime grows in norm $\N(\mathcal{P})$. 

In Section \ref{se:rogers}, we will show that for a large class of functions, the averages over $\mathcal{L}(\mathcal{P},s)$
converge to the 
averages that are predicted by the averages of Haar-random module lattices $\ML_{t}(K)$.

\label{se:average_over_lattice}

This setup can be thought of as a higher moment analogue for general number fields of the following first moment result, which could be attributed to Rogers \cite{Rog47} as seemingly the first to consider an effective version of Siegel's mean value theorem.
\begin{theorem}
{\bf (Rogers, 1947) }
\label{th:rogers}

Let $p$ be an arbitrary prime, $\mathbb{F}_p$ be the field with $p$ elements and let $\pi_p:\mathbb{Z}^{d} \rightarrow \mathbb{F}_p^{d}$ be the natural coordinate-wise projection map.  Let $\mathcal{L}_p$ be the set of sub-lattices of $\mathbb{Z}^{d}$ that are pre-images of one-dimensional subspaces via this projection map, scaled to become unit covolume, i.e. $$\mathcal{L}_p = \{ C_p \pi_p^{-1}( \mathbb{F}_p v) \mid v \in \mathbb{F}_p^d \setminus \{ 0\} \}\text{ with } C_p = p^{-\left(1-\frac{1}{d}\right)} .$$

Consider a compactly supported continuous function $f:\mathbb{R}^{d} \rightarrow \mathbb{R}$. Then the following holds:
\begin{align}
\lim_{ p \rightarrow \infty}\left[ \frac{1}{\card {\mathcal{L}}_p}\sum_{ \Lambda \in {\mathcal{L}}_p} \sum_{v\in\Lambda\setminus\{0\}} f(v) \right] = \int_{\mathbb{R}^{d}} f(x) dx.
\end{align}
\end{theorem}

Theorem \ref{th:rogers} can for instance be used to effectively generate lattices that attain the Minkowski lower bound on the sphere packing density. Results along the line of Theorem \ref{th:rogers} have appeared \cite{M16,GS21} for mean values of lattices with additional structure and with applications in mind.

\par

\subsection{Class of test functions}
  \label{se:test_funcs}

  The functions that are of interest in this theory are compactly supported continuous functions and functions that are indicators of sets with nice boundaries. The following class of functions contains both of these cases and will allow us to prove the relevant generalizations of Theorem \ref{th:rogers}.

  \begin{hypothesis} We call a test function
$g: \mathbb{R}^{d} \rightarrow \mathbb{R}$ 
``admissible'' if it is a 
compactly supported measurable function such that 
the error function 
\begin{equation}
  E(x,r) = \sup_{\|x-y\| \leq r} |g(x)-g(y)|
\end{equation}
satisfies for some $C>0$ and every $r>0$ small enough
\begin{equation}
  \int_{V} E(x,r) dx \leq C \cdot r,
\end{equation}
for any real subspace $V \subseteq \mathbb{R}^{d}$, $V \neq \{ 0 \}$. The integration is happening with the induced Lebesgue 
measure from the inclusion $V \subseteq \mathbb{R}^{d}$ and $C>0$ is required not to depend on $V$. 
\label{hy:admisible}
\end{hypothesis}

\section{Rogers integral formula using lifts of codes}
  \label{se:rogers}

For this section, let us assume that the number field $K$ is fixed. 
We are then interested in the average of lattice sum functions
\begin{equation}
\frac{1}{\card \mathcal{L}(\mathcal{P},s)}
\sum_{ \Lambda \in \mathcal{L}(\mathcal{P},s)} 
\left(\sum_{v \in \Lambda} f(v)\right)^{n}  = 
\frac{1}{\card \mathcal{L}(\mathcal{P},s)} 
\sum_{ \Lambda \in \mathcal{L}(\mathcal{P},s)} 
\left( \sum_{v \in \Lambda^{n}} g(v)\right),
\label{eq:nthmoment}
\end{equation}
where $g(v_1,\dots,v_n)= f(v_1)f(v_2)\dots f(v_n)$.
We perform some manipulations on this sum. Letting $\ind(P)$ denote the indicator function of a proposition $P$, we have that 
\begin{align}
& 
\frac{1}{\card \mathcal{L}(\mathcal{P},s)} \sum_{\Lambda \in \mathcal{L}(\mathcal{P},s)}\left( \sum_{v \in \Lambda^{n}} g(v)\right) 
\\
& = 
\sum_{x \in \mathcal{O}_K^{t \times n} }
g( \beta x)
\left(
\frac{1}{\card \mathcal{L}(\mathcal{P},s)} 
\sum_{\substack{  S \subseteq k_\mathcal{P}^{t} \\  S \simeq k_{\mathcal{P}}^{s} } } 
\ind\left(
\spant
( \pi_{\mathcal{P}}(x_1),\dots,\pi_{\mathcal{P}}(x_n)) \subseteq S
\right) \right). \\
\label{eq:prepoission}
\end{align}

The inner sum is just the probability of a random subspace $S \subseteq k_\mathcal{P}^{t}$ of fixed dimension $s$ containing some given set of points $x_1,x_2,\dots,x_n \in k_\mathcal{P}^{t}$. This probability, other than depending on $\mathcal{P},s$, depends only on the $k_\mathcal{P}$-dimension of the subspace generated by $\pi_\mathcal{P}(x_1),\dots,\pi_\mathcal{P}(x_n)$. This dimension equals the rank of $\pi_\mathcal{P}(x) \in M_{ t \times n }(k_\mathcal{P})$ which is certainly less than the rank of $x\in M_{t \times n}(\mathcal{O}_K) \subseteq M_{t \times n }(K)$. So we can split our sum into 
\begin{align}
& = 
\sum_{m=0}^{\min(n,t)}\sum_{\substack{x \in M_{t \times n }(\mathcal{O}_K) \\ \rank(x) = m}}
\frac{  g( \beta x)}{\card \mathcal{L}(\mathcal{P},s)} 
\left(
\sum_{\substack{  S \subseteq k_\mathcal{P}^{t} \\  S \simeq k_{\mathcal{P}}^{s} } } 
\ind\left(
\spant
( \pi_{\mathcal{P}}(x_1),\dots,\pi_{\mathcal{P}}(x_n)) \subseteq S
\right) \right). 
\label{eq:not_even_the_final_form}
\end{align}

Given $x \in M_{t \times n }(\mathcal{O}_K)$, we might for some $\mathcal{P}$ encounter a ``rank-drop'' phenomenon, that is $\rank\left(\pi_\mathcal{P}(x)\right) < \rank(x)$. However, the good news is that the matrices $x$ where this rank-drop happens can be ``pushed away'' from the support of $g$.

\begin{lemma}
Suppose that $x \in M_{t \times n }\left(\mathcal{O}_{K}\right)$ is a matrix with $\rank(x) = m>0$ and $\mathcal{P}$ is a prime ideal in $\mathcal{O}_K$ such that $\rank(\pi_\mathcal{P}(x)) < m$. Then, for any Euclidean norm $\| \bullet \|: M_{t \times n }(K_\mathbb{R})  \rightarrow \mathbb{R}_{\ge 0}$, we can find a constant $C>0$ depending on $K,\|\bullet\|,t,n$ and independent of $m,\mathcal{P}$ such that 
\begin{equation}
\|x\| \ge C \N(\mathcal{P})^{\frac{1}{m[K:\mathbb{Q}]}}
\end{equation}
\label{le:rankdrop}
\end{lemma}
\begin{proof}
By choosing a $\mathbb{Z}$-basis of $\OK$, we can embed $\iota:\mathcal{O}_K \hookrightarrow M_{[K:\mathbb{Q}]}(\mathbb{Z})$ as a subring of the square integer matrices of size $[K:\mathbb{Q}]$. 
Without loss of generality, we assume that the norm $\|\bullet \|$ is the Euclidean norm via the embedding 
$$\iota: M_{t \times n }(\OK)\hookrightarrow M_{t[K:\mathbb{Q}] \times n[K:\mathbb{Q}]}(\mathbb{Z}) \subseteq \mathbb{R}^{tn[K:\mathbb{Q}]^{2}}.$$ 
Since $\rank(x)=m$, we know that there exists a non-singular $m \times m$ minor $a \in M_{m}(\OK)$ appearing as a submatrix in $x$. It is clear that $ 0 \neq \det a \in \mathcal{P}$ otherwise there is no rank-drop modulo $\mathcal{P}$. Therefore, we get that 
$\N(\mathcal{P}) \mid \N(\det a).$
Since we know that $0 \neq |\det(\iota(a))| \ge \N(\mathcal{P})$, at least one non-zero integer appearing in the matrix entries of $\iota(a)$ would have absolute value $\gg \N(\mathcal{P})^{\frac{1}{m[K:\mathbb{Q}]}}$. This produces the same lower bound on the Euclidean norm of $\iota(a)$ up to a constant, and similarly also for $\iota(x)$.
\end{proof}

\begin{lemma}
\label{le:counting}
Suppose $y_1, y_2,\dots,y_m \in k_{\mathcal{P}}^{t}$ are linearly independent vectors (over $k_{\mathcal{P}}$). Then the following holds:
{\small\begin{equation}
\frac{  1}{\card \mathcal{L}(\mathcal{P},s)} 
\left(
\sum_{\substack{  S \subseteq k_\mathcal{P}^{t} \\  S \simeq k_{\mathcal{P}}^{s} } } 
\ind\left(
\spant
( y_1,y_2,\dots,y_m) \subseteq S
\right) \right) = \begin{cases}
0 &  \text{ if $s<m$}\\
\N(\mathcal{P})^{-m(t-s)}\cdot (1+o(1)) &  \text{ if $s \geq m$.}
\end{cases}
\end{equation}}
Here the error term $o(1)$ is with respect to growing norm $\N(\mathcal{P})$.
\end{lemma}

\begin{theorem}
\label{th:higher_moments}
Take $t \ge 2$, $n \in \{1,\dots,t-1\}$ and choose $s$ as either $t-1$, or any number in $\{n,n+1,\dots,t-1\}$ that satisfies
\begin{equation}
1-\frac{s}{t} < \frac{1}{n}.
\end{equation}

Let $g:K_{\mathbb{R}}^{t \times n} \rightarrow \mathbb{R}$ be a function satisfying Hypothesis \ref{hy:admisible}. With $\mathcal{L}(\mathcal{P},s)$ defined as in (\ref{eq:def_of_L}), we have that as $\N(\mathcal{P}) \rightarrow \infty$
{\small \begin{equation}
\frac{1}{\card \mathcal{L}(\mathcal{P},s)} 
\sum_{ \Lambda \in \mathcal{L}(\mathcal{P},s)} 
\left( \sum_{v \in \Lambda^{n}} g(v)\right)
\rightarrow
\sum_{m=0}^{{n}}\sum_{\substack{D \in M_{m \times n }(K) \\ 
\rank(D) = m \\ 
D \text{ row reduced echelon}}}
\mathfrak{D}(D)^{-t} \int_{x \in K_\mathbb{R}^{t \times m }} g(x D ) dx,
\label{eq:right_side_converge}
\end{equation}}
where $\mathfrak{D}(D)$ is the index of the sublattice $\{ C \in M_{1 \times m }(\mathcal{O}_K) \mid C \cdot D \in M_{1 \times n}(\mathcal{O}_K)\}$ in $M_{1 \times m}(\mathcal{O}_K)$. 
Here the right hand side is known to converge for $t>n$. See \cite{GSV23}.
\end{theorem}
\begin{proof}
From the discussion above, we arrive at (\ref{eq:not_even_the_final_form}), and it remains to consider
\begin{align}
\sum_{m=0}^{n}\sum_{\substack{x \in M_{t \times n }(\mathcal{O}_K) \\ \rank(x) = m}}
\frac{  g( \beta x)}{\card \mathcal{L}(\mathcal{P},s)} 
\left(
\sum_{\substack{  S \subseteq k_\mathcal{P}^{t} \\  S \simeq k_{\mathcal{P}}^{s} } } 
\ind\left(
\spant
( \pi_{\mathcal{P}}(x_1),\dots,\pi_{\mathcal{P}}(x_n)) \subseteq S
\right) \right).
\end{align}
Note that here $\beta = \beta(\mathcal{P},s) = \N(\mathcal{P})^{-\left(1-\frac{s}{t}\right)\frac{1}{[K:\mathbb{Q}]}}$. The rank $m$ ranges within $\{0,1,\dots,n\}$ since $\min(n,t)=n$. Also, since $s \ge n$, we expect the quantity in parentheses to be nonzero.\par
We recall that $M_{t \times n }(K_{\mathbb{R}})$ has the Euclidean measure given by $t \cdot  n$ copies of the quadratic form coming from (\ref{eq:norm}).
When $m>1$, we know from Lemma \ref{le:rankdrop} that we will encounter a rank-drop mod $\mathcal{P}$ only if for some predetermined constant $C>0$
\begin{align}
& \| x \|  \ge C \N(\mathcal{P})^{\frac{1}{m[K:\mathbb{Q}]}}  \\
\Rightarrow &  \|\beta x\| \ge C \N(\mathcal{P})^{\frac{1}{[K:\mathbb{Q}]} \cdot \left( \frac{1}{m} - \left(1 - \frac{s}{t}\right)\right)}.
\end{align}
Since 
\begin{equation}
\frac{1}{m} - \left(1-\frac{s}{t} \right) \ge  \frac{1}{n} - \left(1- \frac{s}{t}\right) > 0,
\end{equation}
for a large enough value of $\N(\mathcal{P})$ we have that all the matrices of $x \in M_{t  \times n }\left(\mathcal{O}_K\right)$ where rank-drop could happen are outside the support of $g$. Let us assume that $\N(\mathcal{P})$ is large enough for this to hold. Hence whenever $g(\beta x)$ is non-zero, the span of $ \pi_\mathcal{P}(x_1) , \pi_\mathcal{P}(x_2) ,\dots,  \pi_\mathcal{P}(x_n) $ is of the same $k_\mathcal{P}$-dimension as the rank of $x$. 
Using Lemma \ref{le:counting}, we can rewrite our sum as
{\small\begin{align}
\sum_{m=0}^{{n}}\sum_{\substack{x \in M_{t \times n }(\mathcal{O}_K) \\ \rank(x) = m}}
\frac{  g( \beta x)}{ \N(\mathcal{P})^{m(t-s)}} \cdot (1+o(1))
= \sum_{m=0}^{{n}}\sum_{\substack{x \in M_{t \times n }(\mathcal{O}_K) \\ \rank(x) = m}}
{  g( \beta x)}{ \beta^{mt[K:\mathbb{Q}]} } \cdot (1+o(1)).
\end{align}}

The result follows as $\beta \rightarrow 0$ from this last step together with Proposition \ref{pr:convergence}.
\end{proof}

\begin{remark}
Observe that the right hand side of (\ref{eq:right_side_converge})
is identical to Rogers' integration formula \cite{K19, hughes2023mean, GSV23}. This is no accident: on the one hand, the integral formula holds for a lattice taken at random on any fixed component of $\ML_t(K)$, see e.g. \cite[Remark 7]{GSV23}. 

On the other hand, the lattices in $\mathcal{L}(\mathcal{P},s)$ lie in the same connected component of $\ML_t(K)$ for a fixed $\mathcal{P}$ and $s$. 
As $\N(\mathcal{P}) \rightarrow \infty$ it is well known that the ideal classes $[\mathcal{P}]$ equidistribute in the class group, so that, fixing the codimension of the algebraic codes by setting e.g. $s=t-1$, the sets $\mathcal{L}(\mathcal{P},s)$ equidistribute among the connected components of $\ML_t(K)$ and approach the integral formula on each component.

Finally, equidistribution results for Hecke points as in \cite{EOL01} should imply the equidistribution of $\mathcal{L}(\mathcal{P},s)$ as $\N(\mathcal{P}) \rightarrow \infty$, so that our results would give a new proof of the (continuous) Rogers integral formula and vice-versa.
\end{remark}

\begin{remark}
One can find the rate of convergence from Proposition \ref{pr:convergence}.
\end{remark}

\subsection{Matrices of fixed rank}

\begin{definition}
For notational simplicity, let $\mathcal{R}_{m,n}(K)\subseteq M_{m \times n}(K)$ denote the set of row-reduced echelon form matrices of $K$-rank $m$. 
\end{definition}
The rest of this section is devoted to the proof of Proposition \ref{pr:convergence} below, which will establish Theorems \ref{thm:KatznelsonOK} and \ref{th:higher_moments}. As mentioned in the introduction, this is very similar to the work done by Katznelson in \cite{K1994}. The proof basically extends the same approach to number fields. \par
Throughout this section, we let $g:M_{t \times n}(K_\mathbb{R}) \rightarrow \mathbb{R}$ denote an admissible function in the sense of Hypothesis \ref{hy:admisible} and write $\gg$ or $\ll$
to denote inequality with an implicit constant dependent on $K,n,m,t$ and the function $g$ only.

\begin{proposition}
\label{pr:convergence}
	Suppose $K$ is a number field with ring of integers $\OK$ and equip any number of copies of $K_\mathbb{R} = K \otimes \mathbb{R}$ with the trace norm normalized so that $\OK$-points have covolume 1. Let $d=[K:\mathbb{Q}]$ and let $g:M_{t \times n}(K_\mathbb{R}) \rightarrow \mathbb{R}$ be an admissible function in the sense of Hypothesis \ref{hy:admisible}. Then as $\beta \rightarrow 0$ we have that 
	\begin{equation}
	  \left| \sum_{\substack{A \in M_{t \times n}(\OK) \\ \rank(A)=m}} g(\beta A) \beta^{mtd}  -  \sum_{D \in \mathcal{R}_{m,n}(K)}  \frac{1}{\mathfrak{D}(D)^{t}}\int_{M_{t \times m}(K_\mathbb{R})} g(x D) dx\right| \ll \beta^\delta
	\end{equation}
	for any $0<\delta<1$, where the notations are as in Theorem \ref{th:higher_moments}.
\end{proposition}
\begin{remark}
\label{re:rate_of_convergence}
It should be possible to improve the error term to $O(\beta)$, however this was not the focus for our current work.
\end{remark}

We recall that the right hand side of the sum in Proposition \ref{pr:convergence} is
finite as long as we are in the regime $t>n$. Let us for the remainder of this section fix a $\mathbb{Z}$-basis of $\OK$, noting that the implicit constants might depend on this choice of $\mathbb{Z}$-basis. We consider the set 
\begin{align}
	V_{t,n,m}(\mathcal{O}_K) & = \{ x \in M_{t \times n}(\OK) \mid \rank x = m  \} \\
	 & = \bigsqcup_{\substack{ D \in \mathcal{R}_{m,n}(K)}} 
 \{  CD \mid C \in M_{t \times m}(\OK) ,   CD \in M_{m \times n}(\mathcal{O}_K) , \rank (C) = m\}.
\end{align}

\begin{definition}
Let $D \in \mathcal{R}_{m,n}(K)$.
Define
\begin{align}
\Lambda_D &  = ( M_{1 \times m}(K)\cdot D ) \cap M_{1 \times n}(\OK).
\end{align}
Hence $\Lambda_D$ is a lattice that contains all the vectors in $M_{1 \times m}(K) \cdot D$ with integer entries. It is an $\OK$-module. We write 
\begin{align}
	\Lambda_D \otimes \mathbb{R} & = \Lambda_D \otimes_{\mathbb{Z}} \mathbb{R} = \Lambda_D \otimes _{\OK} K_{\mathbb{R}},\\
	\Lambda_D \otimes \mathbb{Q} & = \Lambda_D \otimes_{\mathbb{Z}} \mathbb{Q}  = \Lambda_D \otimes_{\OK } K .
\end{align}
\end{definition}

\begin{definition}
	\label{de:height_definition}
For any Euclidean space $V$ and for any lattice $\Lambda \subseteq V$, we define 
\begin{equation}
  H(\Lambda)  = \vol(\Lambda \otimes \mathbb{R} /\Lambda)
\end{equation}
taken with respect to the restriction of the norm to $ \Lambda \otimes \mathbb{R} \subseteq \mathbb{R}^{d}$. Here, by a lattice we mean any closed discrete subgroup of $V$.
\end{definition}

\begin{remark}
The assignment $D \mapsto \Lambda_D \otimes \mathbb{Q}$ 
is a bijection between $\mathcal{R}_{m,n}(K)$ and the rational subspaces of $K$-dimension $m$ in $K^{n}$.
Furthemore, the assignment $D \mapsto \Lambda_D$ is a bijection between $\mathcal{R}_{m,n}(K)$ and primitive $\OK$-modules of rank $m$ inside $M_{1 \times n}(K)$. Here primitive means that the $\OK$-module is not strictly contained in another $\OK$-module of rank $m$ inside $M_{1 \times n}(K)$. When there is no ambiguity, we shall at times write $H(\Lambda_D)$ as $H(S)$ or $H(D)$ for $S=M_{1 \times m}(K) D$.
\end{remark}

\begin{definition}
	\label{de:defi_of_M_t}
	For any subset $R \subseteq M_{1 \times n}(K_\mathbb{R})$, we denote by $M_{t}(R)$ the set of matrices whose rows only contain elements of $R$.
\end{definition}

\begin{remark}
Observe that $M_{t}(\Lambda_D \otimes \mathbb{Q}) = M_{t \times m}(K) \cdot D$ 
and $M_{t}(\Lambda \otimes \mathbb{R}) = M_{t \times m}(K_\mathbb{R}) \cdot D$. However, $M_t(\Lambda_D) \subsetneq  M_{t \times n}(\OK) \cdot D$ in general for $D \in \mathcal{R}_{m,n}(K)$.
\end{remark}

We remark that the covolume of $M_t(\Lambda_D)$ is related to the covolume $H(D)$ of $\Lambda_D$ by the relation$$H(M_t(\Lambda_D))=H(D)^t.$$
This can be seen by choosing a suitable basis from a reduced basis of $\Lambda_D$.
Since the product of the ``Jacobian'' of the map $x \rightarrow x D$ for $x \in M_{1 \times n}(K_\mathbb{R})$ times the factor $\mathfrak{D}(D)$ is exactly $H(D)$, 
one can show (see Appendix A of \cite{GSV23})
\begin{equation}
  \frac{1}{\mathfrak{D}(D)^{t}} \int_{M_{t\times m}(K_\mathbb{R})} g( x D)dx = \int_{M_t( \Lambda_D )} g(x)d_{D}x,
  \label{eq:d_D_defined}
\end{equation}
where $d_D x$ is a Lebesgue measure on $M_t(\Lambda_D)$ such that $M_t(\Lambda_D) \subseteq M_t(\Lambda_D \otimes \mathbb{R})$
has unit covolume.

One can decompose the set
\begin{equation}
	V_{t,n,m}(\mathcal{O}_K)  = \bigsqcup_{\substack{D \in \mathcal{R}_{m,n}(K)}} 
	\{A \in M_t(\Lambda_D) \mid \rank(A) = m \},
\end{equation}
where $M_t(\Lambda_D)$, as described in Definition \ref{de:defi_of_M_t}, is defined as the set of those matrices in $M_{t \times n}(\OK)$ whose rows are made of elements of $\Lambda_D$.

\subsubsection{How many matrices are we summing up?}

One necessary condition on matrices $D \in \mathcal{R}_{m,n}(K)$ contributing to the count is the existence of at least one full-rank $A \in M_{t}(\Lambda_D)$ such that $g(A \beta) \neq 0$, that is $\|A\|\ll \tfrac{1}{\beta}$. To deal with this, let us introduce the following notation.

\begin{definition}
\label{de:def_of_mathcalR_beta}
We  define 
\begin{equation}
  \mathcal{ R }^{\beta} _{m,n}(K) = \{ D \in \mathcal{R}_{m,n}(K) \mid   \exists A \in M_t(\Lambda_D), \rank A = m, \|A\| \ll \tfrac{1}{\beta}  \}.
\end{equation}
Here the implicit constant in $\|A\| \ll \tfrac{1}{\beta}$ 
is fixed to ensure that $g(A \beta) = 0$ when $\|A \|\gg \frac{1}{\beta}$.
\end{definition}

The following is a result of W. Schmidt \cite{S1967}:
\begin{theorem}
\label{th:schmidt}
\begin{equation}
	T^{n} \ll \card \{ D \in \mathcal{R}_{m,n}(K) \mid H(D) \leq T \} \ll T^{n}.
\end{equation}
\end{theorem}
In fact, more precise asymptotics were established by J. Thunder \cite{T1993}, but we will not require those.
We estimate an upper bound on the height for the finitely many $D\in\mathcal{R}_{m,n}^{\beta}(K)$:
\begin{lemma}
\label{le:lower_bound_on_C}
For $D \in \mathcal{R}_{m,n}(K)$, 
we have that if $A \in M_{t}(\Lambda_D)$ 
and $\rank (A) = m$
then
\begin{equation}
	\|A \|^{m[K:\mathbb{Q}]} \gg H(D).
\end{equation}
\end{lemma}
\begin{proof}
The $t$ rows of $A$ contain $m$ $K$-linearly independent rows in $\Lambda_D$. 
Hence, if we look at the $\OK$-module generated by the rows of $A$, this will generate
a sublattice $\Lambda$ within $\Lambda_D$. Therefore 
\begin{equation}
	H(\Lambda) \geq H(D).
\end{equation}

By Minkowski's second theorem or by knowing that the Hadamard ratio of a lattice is bounded from below, this shows that any $\mathbb{Z}$-basis of $\Lambda$ should have their products of lengths $\gg H(D)$.
Then, using a $\mathbb{Z}$-basis of $\OK$ multiplied with the columns of $A$, we can can obtain a $\mathbb{Z}$-basis
$\{l_1^{\Lambda},l_2^{\Lambda},\dots,l_{md}^{\Lambda}\}$ of $\Lambda$. 

We know by the arithmetic-mean geometric-mean inequality that 
\begin{equation}
  \sum  \| l_i^{\Lambda}\|^{2} \gg \prod \|l_i^{\Lambda}\|^{\frac{2}{md}} \gg H(D)^{\frac{2}{md}}.
\end{equation}
The statement now follows 
from relating $\|A\|^{2}$ to $\sum \|l_{i}^{\Lambda}\|^{2}$.
\end{proof}

\begin{corollary}
We have 
\begin{equation}
	D \in \mathcal{R}^{\beta}_{m,n}(K) \Rightarrow H(D) \ll \tfrac{1}{\beta^{m[K:\mathbb{Q}]}}.
\end{equation}
Therefore, 
\begin{equation}
  \sum_{D \in  \mathcal{R}^{\beta}_{m,n}(K)} \sum_{\substack{A \in M_{t}(\Lambda_D)   \\ \rank A = m  }} g(\beta A )  \beta^{mtd} = 
  \sum_{\substack{D \in \mathcal{R}_{m,n}(K)   \\  H(D) \ll \frac{1}{\beta^{md}}}} \sum_{\substack{A \in M_{t}(\Lambda_D)   \\ \rank A = m  }} g(\beta A )  \beta^{mtd}.
\end{equation}
\label{co:calprime_bounded}
\end{corollary}

\begin{remark}
\label{re:necessary_not_sufficient}

Therefore, Corollary \ref{co:calprime_bounded} along with Thereom \ref{th:schmidt} 
tells us that  $\card \mathcal{R}_{m,n}^{\beta}(K)$ is $O(\frac{1}{\beta^{mnd}})$. The converse of Corollary \ref{co:calprime_bounded},
however, is not true. This is in fact one of the technical difficulties in proving Proposition \ref{pr:convergence}.
\end{remark}


To prove Proposition \ref{pr:convergence}, 
we claim that it only needs to be shown that 
\begin{equation}
\label{eq:terms_to_focus_on}
	\sum_{D \in \mathcal{R}^{\beta}_{m,n}(K)}
	\left| \int_{M_t(\Lambda_D \otimes \mathbb{R})} g(x)d_{D}x  - \sum_{\substack{A \in M_t(\Lambda_D) \\ \rank(A) = m}}  \beta^{mtd}g(\beta A) \right|  \rightarrow 0  \text{ as $\beta \rightarrow 0$ }.  
\end{equation}

Here the integral on $M_t(\Lambda_D \otimes \mathbb{R})$ is against a measure  $d_Dx$
that renders $M_t(\Lambda_D)\subseteq M_t(\Lambda_D \otimes \mathbb{R})$ of unit covolume. 
We shall need the following lemma to dispatch the terms not considered in (\ref{eq:terms_to_focus_on}).
\begin{lemma}
\label{le:everything_else}
Suppose that $D \in \mathcal{R}_{m,n}(K) \setminus \mathcal{R}_{m,n}^{\beta}(K)$. Then
we get 
\begin{equation}
  H(D) \gg \tfrac{1}{\beta^{d}}.
\end{equation}
\end{lemma}
\begin{proof}
	If there is not a single $A \in M_t(\Lambda_D)$ with $\rank (A) = m$ and $\|A\| \ll \frac{1}{\beta}$, then this means that it is impossible to find $m$ $K$-linearly independent vectors in $\Lambda_D$ with lengths $\ll \tfrac{1}{\beta}$. 
	However, in Lemma \ref{le:props_of_minima}, we will see that $\Lambda_D$ has $m$ linearly independent vectors $l_1,l_2,\dots,l_m$ such that 
	\begin{equation}
	  \|l_1\| \|l_2\| \dots \|l_m\| \ll H(D)^{\frac{1}{d}}.
	\end{equation}
	We know that all the $\|l_i\| \gg 1$ because $l_i \in \Lambda_D \subseteq M_{1 \times n}(\OK)$. And by the discussion above, we should have at least one $\|l_i\| \gg \frac{1}{\beta}$. This yields the desired result.
\end{proof}

Hence, we want to show that as $\beta \rightarrow 0$, we have the convergence of the terms
\begin{equation}
	\sum_{\substack{D \in \mathcal{R}_{m,n}(K)  \\ H(D) \gg \frac{1}{\beta^{d}} }} \int_{M_t( \Lambda_D \otimes \mathbb{R} )} g(x) d_D x \rightarrow 0.
\label{eq:tail_of_zeta}
\end{equation}
But this is true since they are the tail in the convergent sum
\begin{equation}
	\sum_{\substack{D \in \mathcal{R}_{m,n}(K)  }} \int_{M_t(\Lambda_D \otimes \mathbb{R} )} g(x) d_D x  < \infty.
\end{equation}
Using Theorem \ref{th:schmidt} and that $\int_{M_t(\Lambda \otimes \mathbb{R})} g(x) d_D x \ll H(D)^{-t}$, one can conclude that the rate of convergence in (\ref{eq:tail_of_zeta}) is at least $O(\beta)$. This establishes our claim that we only need to show (\ref{eq:terms_to_focus_on}).

\subsubsection{Artificially injecting lower rank matrices}

For $D \in \mathcal{R}_{m,n}(K)$ and $v \in M_t(\Lambda_D)$, it is clear that $\rank v \leq m$.
Observe that we can rewrite
\begin{align}
 & \left| \int_{M_t(\Lambda_{D} \otimes \mathbb{R})}g ( x)d_{D} x  - \sum_{\substack{v \in M_t(\Lambda_D) \\ \rank v = m}} g(\beta v ) \beta^{mtd} \right|
 \\ 
 & \leq 
 \left| \int_{M_t(\Lambda_D \otimes \mathbb{R})}  g ( x)d_{D} x  - \sum_{\substack{v \in M_t(\Lambda_D) }} g(\beta v ) \beta^{mtd} \right| + 
 \sum_{\substack{v \in M_t(\Lambda_D) \\ \rank v < m}} \left|g(\beta v)\right| \beta^{mtd} . 
 \label{eq:left_and_right_terms}
\end{align}

The terms on the right side of (\ref{eq:left_and_right_terms})
are lower rank terms which we will deal with in Section \ref{se:lower_rank}. When we sum up the terms on the left, we get 
\begin{equation}
 \sum_{\substack{D \in \mathcal{R}^{\beta}_{m,n}(K)  }}
 \left| \int_{M_t(\Lambda_D \otimes \mathbb{R})}g ( x)d_{D} x  - \sum_{\substack{v \in M_t(\Lambda_D) }} g(\beta v ) \beta^{mtd} \right|.
\end{equation}
The quantity being summed is precisely the difference between a Riemann integration and a Riemann sum. Let us now show that this difference eventually vanishes as $\beta \rightarrow 0$.

\subsubsection{Error estimates for the difference between the integrals and Riemann sums}

Here's a lemma that lets us exploit our hypothesis on the class of test functions.

\begin{lemma}
\label{le:covering_radius}
Let $g:\mathbb{R}^{d} \rightarrow \mathbb{R}$ be a function that is admissible in the sense of Hypothesis \ref{hy:admisible}. Let $\Lambda \subseteq \mathbb{R}^{d}$ be a lattice of $\mathbb{Z}$-rank $m$. 

Then
  \begin{equation}
	  \left|  \int_{\Lambda\otimes \mathbb{R}} g(x)dx  - H(\Lambda)\sum_{v \in \Lambda}g(\beta v) \beta^{m} \right| \ll \beta c(\Lambda)
  \end{equation}
  where $c(\Lambda)$ is the covering radius of $\Lambda$ and $H(\Lambda)$ is the height of $\Lambda$ as defined in Definition \ref{de:height_definition}.  Here the integration on $\Lambda \otimes \mathbb{R}$ 
  is with respect to the Lebesgue measure induced by the inclusion in $\mathbb{R}^{d}$, which is also used to define $H(\Lambda)$.
\end{lemma}
\begin{proof}
	Let $F \subseteq \Lambda_\mathbb{R}$ be a fundamental domain 
	of $\Lambda_\mathbb{R}/ \Lambda$ such that $F \subseteq B_{c(\Lambda)}(0)$. Then, we can write 
  {\small\begin{align}
	  \left|  \int_{x \in \Lambda\otimes \mathbb{R}} g(x)dx  - \left(\int_{F} dx\right) \sum_{v \in \Lambda}g(\beta v) \beta^{d} \right|
	  & = 
	  \left|  \int_{x \in \Lambda\otimes \mathbb{R}} g( \beta x) \beta ^{m } dx  - \left(\int_{F} dx \right) \sum_{v \in \Lambda}g(\beta v) \beta^{m} \right| \\
	  & = 
	  \left|  \sum_{v \in \Lambda } \int_{F} g( \beta x+ \beta v) \beta^{m} dx  - \sum_{v \in \Lambda} \int_{F} g(\beta v)  \beta^{m} dx \right| \\
	  & \leq
	    \sum_{v \in \Lambda } \int_{F} \left|g( \beta x+ \beta v)   -g(\beta v) \right| \beta^{m} dx. \\
  \end{align}}
  Let $E(x,r)$ be the error function associated to $g$ as in Hypothesis \ref{hy:admisible}. Then the error function associated to $x \mapsto g(\beta x)$ is $x \mapsto E( \beta x,{\beta} r)$ since 
  \begin{equation}
    \sup_{\|x-y\| < r} |g(\beta x) - g(\beta y)| 
    =
    \sup_{\| \beta x - y \| < \beta r  } |g(\beta x) - g(y)| .
  \end{equation}
  Therefore 
  \begin{align}
    \sum_{v \in \Lambda } \int_{F} \left|g( \beta x+ \beta v)   -g(\beta v) \right| \beta^{m} dx  
    & \leq \sum_{v \in \Lambda} \int_{F} E(\beta v + \beta x, \beta c(\Lambda)) \beta^{m}dx \\
    & = \int_{\Lambda\otimes {\mathbb{R}}} E(\beta x,\beta c(\Lambda))\beta^{m} dx  \\
    & = \int_{\Lambda\otimes {\mathbb{R}}} E( x,\beta c(\Lambda))dx \\
    & \ll \beta c(\Lambda).
  \end{align}
\end{proof}

\begin{corollary}
The following estimate holds for the difference between the integrals and the Riemann sums:
\begin{equation}
  \left |  \frac{1}{H(D)^{t}}   \int_{M_t(\Lambda_D \otimes \mathbb{R})} g(x) dx \  -\sum_{v \in M_t(\Lambda_D)} g(\beta v) \beta^{mtd} \right| \ll \beta \frac{c(\Lambda_D)}{ H(D)^{t}},
\end{equation}
where the integral on $M_{t}(\Lambda_D \otimes \mathbb{R})$ is with respect to the induced Lebesgue measure from $M_t( \Lambda_D \otimes \mathbb{R} ) \subseteq M_{t \times n}(K_\mathbb{R})$.
\end{corollary}
\begin{proof}
The only thing to check is that $c(M_t(\Lambda_D))\ll c(\Lambda_D)$.
\end{proof}

\begin{remark}
For the integral on $M_{t}(\Lambda_D \otimes \mathbb{R})$ with respect to the induced Lebesgue measure, one can also write
\begin{equation}
\frac{1}{H(D)^{t}}   \int_{M_t(\Lambda_D \otimes \mathbb{R})} g(x) dx
= 
\int_{M_{t}(\Lambda_D \otimes \mathbb{R})} g(x) d_D x,
\end{equation}
where $d_D x$ is as described in (\ref{eq:d_D_defined}).
\end{remark}


Therefore, in order to prove Proposition \ref{pr:convergence} we have to show that as $\beta \rightarrow 0$, apart from the 
terms to be discussed in Section \ref{se:lower_rank}, the following happens:
\begin{equation}
  \sum_{\substack{D \in \mathcal{R}_{m,n}(K) \\ \rank D = m  \\   H(D) \ll \frac{1}{\beta^{md}}}}  \beta\frac{c(\Lambda_{D})}{H(D)^{t}} \rightarrow 0.
\end{equation}

We will split the sum into two sums based on how ``nice'' $D$ is. 

Let $\delta \in~ ]0,1[$.
For a given $D \in \mathcal{R}_{m,n}(K)$, the following two things could happen with respect to the covering radius $c(\Lambda_{D})$: 
\begin{enumerate}
  \item 
Either $c(\Lambda_{D}) \ll  \frac{1}{\beta^{1-\delta}} \leq \frac{1}{\beta}$, which we will call the well-rounded case,
\item
or we have $  c(\Lambda_{D}) \gg   \tfrac{1}{\beta^{1-\delta}}$ which we will call the skew case. 
\end{enumerate}
The ideal choice of $\delta$ is as close to 1 as possible.

{\bf Well-rounded case:}

For this, we simply observe that
\begin{equation}
	\sum_{\substack{D \in \mathcal{R}_{m,n}(K) \\   H(D) \ll \frac{1}{\beta^{md}} \\ D \text{ well-rounded}}}  \beta\frac{c(\Lambda_{D})}{H(D)^{t}}  \ll
	\sum_{\substack{D \in \mathcal{R}_{m,n}(K) \\   H(D) \ll \frac{1}{\beta^{md}}  }  }  \frac{\beta^{\delta} }{H(D)^{t} }   \ll \beta^{\delta}.
\end{equation}
This is because we know that Schmidt's Theorem \ref{th:schmidt} and partial summation that 
\begin{equation}
	\sum_{\substack{D \in \mathcal{R}_{m,n}(K)  }}  \frac{1}{H(D)^{t}}   < \infty.
\end{equation}

In order to have the rate of convergence as claimed, we should have $\delta$ arbitrarily close to 1.


{\bf Skew case:}

For the skew case, 
we will have to employ some reduction theory of $\OK$-modules as follows: 
\begin{lemma}
	\label{le:props_of_minima}
	Let $d = [K:\mathbb{Q}]$.
Let $\Lambda \subseteq M_{1 \times n}(K)$ be an $\OK$-module of rank $m$. 
Define the successive $K$-minima of $\Lambda$ as
$l_i = l_i(\Lambda)$ for 
$i=1,\dots,m$ given by 
\begin{align}
	l_1(\Lambda) &  = \argmin_{v \in \Lambda \setminus \{0\} } \|v\| \\
  l_2(\Lambda)  & = \argmin_{v \in \Lambda \setminus K\cdot l_1  } \|v\| \\
  l_3(\Lambda)  & = \argmin_{v \in \Lambda \setminus K \cdot l_1 + K \cdot l_2  } \|v\| \\
		&~\vdots
\end{align}
Then, the following statements hold.
\begin{enumerate}
	\item \label{lepart:1}
		Let $H(\Lambda)  = \vol(\Lambda\otimes \mathbb{R}/ \Lambda)$ as before, measured with respect to the canononical inclusion $\Lambda\otimes \mathbb{R} \subseteq M_{1 \times n}(K_\mathbb{R})$.
   We have the relations
   \begin{equation}
     \|l_1\| \ll H(\Lambda)^{\frac{1}{dm}}
   \text{ and }
 \|l_1\|^{d} \dots  \|l_{m}\|^{d} \ll H(\Lambda).
\end{equation}

\item 
	\label{lepart:2}
	We get that 
	\begin{equation}
	  \|l_{m}(\Lambda)\| \gg c(\Lambda) .
	\end{equation}

\item \label{lepart:3}
	For $i<j$, denote the map $\pi_{i}:M_{1\times n}(K_\mathbb{R}) \rightarrow K_\mathbb{R} \cdot l_i$ 
	to be the orthogonal projection onto $K_\mathbb{R} \cdot l_i$. Then
	\begin{equation}
	  \| \pi_{i}(l_j)\| \ll \|l_i\|.
	\end{equation}

\end{enumerate}
  
\end{lemma}
\begin{proof}

{ \em Proof of \ref{lepart:1}: }

This first statement is basically stating the existence of Hermite's constant.
The second is Minkowski's lemma about successive minima adapted to number fields. 

{ \em Proof of \ref{lepart:2}: }

We know that $\Lambda' = \OK l_1 + \dots + \OK l_m$ is a sublattice inside $\Lambda$. Although it is not true that $\Lambda = \Lambda'$ in general, we can say that  $c(\Lambda) \leq c(\Lambda')$.
So it is sufficient to show that $c(\Lambda') \ll \|l_m\|$.

To do this, we can use a $\mathbb{Z}$-basis of $\OK$ to construct from $l_1,\dots,l_m$ 
a $\mathbb{Z}$-basis $l'_{1},l'_{2}, \dots , l'_{md}$ of $\Lambda'$. Without loss of generality, assume that $\|l_1'\| \leq \|l_2'\| \leq \dots  \leq \|l'_{md}\|$. Then, since $ \|l'_{md}\| \ll \|l_{m}\|$,
we know that it is sufficient to show that $c(\Lambda')\ll \|l'_{md}\|$.
This is a standard inequality about the covering radius. See \cite{C12}.

{ \em Proof of \ref{lepart:3}: }

Observe that $l_j + \OK \cdot l_i \subseteq \Lambda$. We also know that for any $\alpha \in \OK$,
the definition of $l_j$ implies $\|l_j\| \leq \|l _j + \alpha \cdot l_i\|$.
It is clear that 
\begin{equation}
  \pi_{i}(l_j + \alpha  \cdot l_i) = \pi_{i}(l_j) +   \pi_{i}(\alpha \cdot l_{i}) .
\end{equation}
Now $\alpha l_{i} \in \OK \cdot l_{i} \subseteq K_\mathbb{R} \cdot l_{i}$ so $\pi_{i}(\alpha \cdot l_i) = \alpha \cdot l_{i}$. Furthermore, we also know that for any $x \in M_{1 \times n}(K_\mathbb{R})$
\begin{equation}
	\|x\|^{2} = \|\pi_{i}(x)\|^{2} + \|\pi_{i}^{\perp}(x)\|^{2},
\end{equation}
where $\pi_{i}^{\perp}$ is the projection to the orthogonal complement of $K_\mathbb{R}\cdot l_{i} \subseteq M_{1 \times n}(K_\mathbb{R})$. 
We know that $\pi_{i}^{\perp}(l_j + \alpha \cdot l_i) = \pi_{i}^{\perp}(l_j)$.
The net result is that 
\begin{align}
& \|\pi_{i}(l_j) + \alpha \cdot l_i \|^{2} +  \|\pi_{i}^{\perp}(l_i)\|^{2}    = \|l_j + \alpha \cdot l_i\|^{2} \geq  \|l_j\|^{2} = \|\pi_{i}(l_j)\|^{2} + \|\pi_{i}^{\perp}(l_j)\|^{2}\\
 \Rightarrow
& \|l_j + \alpha \cdot l_i\|^{2} - \|l_j\|^{2} = \|\pi_{i}(l_j) + \alpha \cdot l_i\|^{2} - \|\pi_{i}(l_{j})\|^{2} \geq 0.
\end{align}
This tells us that
\begin{align}
  & \pi_{i}(l_j) = \argmin_{\alpha \in \OK} \|\pi_{i}(l_j) + \alpha \cdot l_i\| \\
  \Rightarrow & \|\pi_{i}(l_j)\| \leq c( \OK \cdot l_i).
\end{align}
So we now only need to show that the covering radius $c(\OK \cdot l_i) \ll \|l_i\|$. This follows 
from the proof of Part \ref{lepart:2} of the statement.
\end{proof}

We are now ready to evaluate our sum using the lemma and obtain:
\begin{equation}
	\sum_{\substack{D \in \mathcal{R}_{m,n}(K) \\ D \text{ skew} }}
\beta	\frac{  c(\Lambda_D) }{ H(D)^{t} } \ll 
\beta \sum_{\substack{D \in \mathcal{R}_{m,n}(K)  \\  D \text{ skew} }}
	\frac{ \|l_{m}(\Lambda_D) \| }{ \|l_1(\Lambda_D)\|^{td} \cdot \|l_2(\Lambda_D)\|^{td}\dots \|l_m(\Lambda_D)\|^{td}}.
\end{equation}

We know that each $D$ uniquely corresponds to the successive $K$-minima 
$$l_1,\dots,l_m \in \Lambda_D \subseteq M_{1\times n}(\OK).$$
Indeed, one can recover $\Lambda_D$ and therefore $D$ from the successive $K$-minima. We also know that $\|l_1(\Lambda_D)\| \ll H(D)^{\frac{1}{mtd}} \ll \tfrac{1}{\beta}$. Hence, we can estimate the above sum
by summing over all possible
lattices $l_1,\dots,l_m \in M_{1 \times n}(\OK)$.
By Lemma \ref{le:props_of_minima} Part \ref{lepart:3}, we get that 
{\small\begin{align}
	\sum_{\substack{D \in \mathcal{R}_{m,n}(K) \\  D \text{ skew} }}
\beta	\frac{  c(\Lambda_D) }{ H(D)^{t} } \ll 
	\sum_{ \substack{l_1 \in M_{1 \times n}(\OK) \\ \|l_1\| \ll \tfrac{1}{\beta}   }}  
	\sum_{ \substack{l_2 \in M_{1 \times n}(\OK) \\ \|l_1\| \ll \|l_2\|  \\ \|\pi_{1}(l_2)\| \ll \|l_1\|   }}  
	\dots
	\sum_{ \substack{l_m \in M_{1 \times n}(\OK)   \\ \|\pi_1(l_m)\| \ll \|l_1\| \\ \|l_m\| \gg \frac{1}{\beta^{1-\delta}} }}  
	\frac{\beta}{\|l_1\|^{td} \|l_2\|^{td} \dots \|l_m\|^{td-1} }. \\
\end{align}}

The following result on summation by parts will come in handy: 
\begin{lemma}
\label{le:domain_dimension_bound}
Let $h_1,h_2 \in \mathbb{Z}_{\geq 1}$. 
Let $D \subseteq M_{1 \times n}(K_\mathbb{R})$ be a domain of infinite volume such that for $T \rightarrow \infty$
\begin{equation}
  \card D \cap M_{1 \times n}(\OK) \cap B_{T}(0) \ll C \cdot T^{h_1},
\end{equation}
where $C$ is a variable and the $\ll$ does not depend on $C$.

Then, 
\begin{equation}
  \sum_{\substack{ l \in  D \cap M_{1 \times n}(\OK) \cap B_{T}(0) \\ a \leq \|l\| \leq b}} \frac{1}{\|l\|^{h_2}} \ll C \left( a^{h_1-h_2} + b^{h_1-h_2} + \int_{a}^{b} t^{h_1 - h_2 -1} dt \right).
\end{equation}
\end{lemma}
\begin{proof}
Summation by parts. See \cite[(13)]{K1994} for details. 
\end{proof}

We will employ Lemma \ref{le:domain_dimension_bound} with 
\begin{equation}
  D = \{ x\in M_{1 \times n}(K_\mathbb{R}) \mid \|\pi_1(x)\| \ll \|l_1 \|\}.
\end{equation}
We then get that $C  = \|l_1\|^{d}$, $h_1 = nd-d$. We put $h_2 = td-1 $, $a=\tfrac{1}{\beta^{1-\delta}}$ and $b = \infty$. The choice of $b$ works because $h_1 - h_2 < 0$ since $t>n$.

This tells us that 
{\smaller\begin{align}
	& \sum_{\substack{D \in \mathcal{R}_{m,n}(K) \\   D \text{ skew} }}
\beta	\frac{  c(\Lambda_D) }{ H(D)^{t} } \\
	\ll & \beta
	\sum_{ \substack{l_1 \in M_{1 \times n}(\OK)  \\ \|l_1\| \ll \tfrac{1}{\beta}  }}  \frac{1}{\|l_1\|^{ td }}
	\sum_{ \substack{l_2 \in M_{1 \times n}(\OK) \\ \|l_1\| \ll \|l_2\|  \\ \|\pi_{1}(l_2)\| \ll \|l_1\|   }}  
	\frac{1}{\|l_2\|^{td}}
	\cdots
	\sum_{ \substack{l_{m-1} \in M_{1 \times n}(\OK)  \\ \|\pi_1(l_{m-1})\| \ll \|l_1\| \\ \|l_{m-1}\| \gg \|l_{m-2}\| }}  
	\frac{\|l_1\|^{d} }{ \|l_{m-1}\|^{td} }
	\left( \frac{1}{\beta^{(1-\delta)}}\right)^ {((n-t-1)d + 1)} \\
	=
	    & \beta^{ 1+ (1-\delta)((t-n+1)d - 1)  )}
	\sum_{ \substack{l_1 \in M_{1 \times n}(\OK)  \\ \|l_1\| \ll \tfrac{1}{\beta}  }}  \frac{1}{\|l_1\|^{ (t-1)d }}
	\sum_{ \substack{l_2 \in M_{1 \times n}(\OK) \\ \|l_1\| \ll \|l_2\|  \\ \|\pi_{1}(l_2)\| \ll \|l_1\|   }}  
	\frac{1}{\|l_2\|^{td}}
	\cdots
	\sum_{ \substack{l_{m-1} \in M_{1 \times n}(\OK)  \\ \|\pi_1(l_{m-1})\| \ll \|l_1\| \\ \|l_{m-1}\| \gg \|l_{m-2}\| }}  
	\frac{1}{ \|l_{m-1}\|^{td} }
	.
\end{align}}

In the innermost sum, let us replace the condition $\|\pi_{1}( l_{m-1} )\| \ll \| l_{1}\|$ with the condition
$\|\pi_{m-2}( l_{m-1} )\| \ll \| l_{m-2}\|$ which holds when $l_1,\dots,l_m$ are the successive $K$-minima of some $\OK$-module.

Then, again using Lemma \ref{le:domain_dimension_bound} with $C = \|l_{m-2}\|^{d}$, $h_1 = nd-d$ and $h_2 = td $, we obtain that the innermost sum is $\ll \tfrac{1}{\|l_{m-2}\|^{d(t-n)}} \ll 1$.
Proceeding in this way all the nested sums can be collapsed. All that remains to show is 
\begin{align}
	\sum_{ \substack{l_1 \in M_{1 \times n}(\OK)  \\ \|l_1\| \ll \tfrac{1}{\beta}  }}  \frac{1}{\|l_1\|^{ (t-1)d }}
	\ll 
	    & o\left(\frac{1}{\beta^{1+ (1-\delta)((t-n+1)d - 1)} }\right) .
\end{align}

Observe that the left hand side is $\ll 1$ if $t>n+1$ and $ \log \frac{1}{\beta}$ for $t=n+1$. So as long as $\delta < 1$, we have the required assertion for the skew case.
\par
\subsection{Lower rank matrices}
\label{se:lower_rank}

So far, we have a proof of Proposition \ref{pr:convergence} when $m=1$ since $m'<m \Rightarrow m' = 0$ and then everything is trivial. We can then assume $m>1$.

We are left with estimating the terms
 \begin{equation}
 \sum_{\substack{D \in \mathcal{R}_{m,n}^{\beta}(K) }}
 \sum_{\substack{v  \in  M_t(\Lambda_D) \\ \rank v < m}} \left|g(\beta v)\right| \beta^{mtd} .
 \end{equation}

Since $g$ is compactly supported,$$\{A\in M_t(\Lambda_D):g(A\beta)\neq0\}\subseteq\{A\in M_t(\Lambda_D):\lVert A\rVert\ll\tfrac{1}{\beta}\}.$$We view the lattice $M_t(\Lambda_D)$ as a lattice in $\R^{tnd}$ of rank $mt$. By the covering radius argument, we obtain that
\begin{equation}\left|M_t(\Lambda_D)\cap B_{1/\beta} \right|\asymp_{m,t,d}
\left(
\frac{1}{\beta } + O(1)
\right)^{mtd} \cdot
\frac{1}{H(D)^t}.
\end{equation}

The contribution from matrices of rank less than $m$ involves matrices whose rows lie in the intersection of two or more distinct $\Lambda_D$'s, with $D\in\mathcal{R}_{m,n}^{\beta}(K)$. It follows by using the inclusion-exclusion principle that

{\small $$
\sum_{D\in\mathcal{R}_{m,n}^{\beta}(K)}\underset{\rank v<m}{\sum_{v\in M_t(\Lambda_D)}}|g(\beta v)|\beta^{mtd}\\=\sum_{1\le s\le m-1}\sum_{D'\in\mathcal{R}_{s,n}^{\beta}(K)}\sum_{v\in M_t(\Lambda_{D'})}|g(\beta v)|\beta^{mtd}n^{\beta}_{m,n}(D',K),
$$where$$n^{\beta}_{m,n}(D',K)=\underset{\Lambda_{D'}\subseteq\Lambda_D}{\sum_{D\in\mathcal{R}_{m,n}^{\beta}(K)}}1.$$}

\begin{lemma}\label{lemma1} Let $1\le s\le m-1$ and let $D'\in\mathcal{R}_{s,n}^{\beta}(K)$. Then$$n^{\beta}_{m,n}(D',K)\ll\frac{1}{\beta^{d(n-s)(m-s))}}\cdot H(D')^{m-s}.$$

\end{lemma}
\begin{proof}
	Let $D_1,D_2\in \mathcal{R}_{m,n}^{\beta}(K)$. Then $\Lambda_{D_i}$ contains a reduced $K$-basis of $\Lambda_{D_i}\otimes\Q$ made of vectors of length $\ll\tfrac{1}{\beta}$. In particular, there exist primitive vectors $(l_j^{(i)})_{j=1}^{m-s}$ in $\mathcal{O}_K^n$, with $\lVert l_j^{(i)}\rVert\ll\tfrac{1}{\beta}$ for all $i,j$, such that$$\Lambda_{D_i}\otimes\Q=(\Lambda_{D'}\otimes\Q)\bigoplus_{j=1}^{m-s}(l_j^{(i)}\cdot K), $$for $i=1,2$. Observe that $\Lambda_{D_1}\otimes\Q=\Lambda_{D_2}\otimes\Q$ if and only if the two $K$-spaces spanned by $(l_j^{(1)})_{j=1}^{m-s}$ and $(l_j^{(2)})_{j=1}^{m-s}$,  respectively, are equal modulo $\Lambda_{D'}\otimes\Q$.

We therefore bound he number of choices for each $l_i$ up to $\Lambda_{D'}$-equivalence. To that end, we bound the number of lattice points ins the ball $B_{1/\beta}$ inside the projection of $\OK^n$ onto $(\Lambda_{D'} \otimes \mathbb{R})^\perp$. This is a lattice in dimension $d(n-s)$ of determinant $H(D')^{-1}$ so that number of choices for one $l_i$ is given by $$\left(\frac{1}{\beta}+O(1)\right)^{d(n-s)} \times \frac{1}{H(D')^{-1}}.$$
Since we are choosing $m-s$ vectors we arrive at the upper bound
\begin{equation}
n^\beta_{n,m}(D') \ll H(D')^{(m-s)} \left(\frac{1}{\beta} +O(1) \right)^{d(n-s)(m-s)}.
\end{equation}
Note that due to possible linear dependencies between the $l_i$, this is an overcount, but it suffices for our purposes. 

\end{proof}

Fix $1\le s \le  m-1$. The summand corresponding to counting lattices of rank $s$ is, by Lemma \ref{lemma1},$$
\sum_{D'\in\mathcal{R}_{s,n}^{\beta}(K)}\sum_{v\in M_t(\Lambda_{D'})}|g(\beta v)|\beta^{mtd}n^{\beta}_{m,n}(D',K)$$

$$\ll \underset{H(D')\ll\tfrac{1}{\beta^{ds}}}{\sum_{D'\in\mathcal{R}_{s,n}(K)}}H(D')^{m-s}\beta^{d(mt-(n-s)(m-s))}\cdot\sum_{v\in M_t(\Lambda_{D'})}|g(\beta v)|.
$$

The inner sum is, again by a covering radius argument,$$\sum_{v\in M_t(\Lambda_{D'})}|g(\beta v)|=\underset{\lVert v\rVert\ll\tfrac{1}{\beta}}{\sum_{v\in M_t(\Lambda_{D'})}}|g(\beta v)|\ll\frac{1}{\beta^{std}\cdot H(D')^t}.$$

Therefore

\begin{align}
& \sum_{D'\in\mathcal{R}_{s,n}^{\beta}(K)}\sum_{v\in M_t(\Lambda_{D'})}|g(\beta v)|\beta^{mtd}n^{\beta}_{m,n}(D',K)\\
\ll & \beta^{d(mt-st-(n-s)(m-s))}
\underset{H(D')\ll
\tfrac{1}{\beta^{ds}}}{\sum_{D'\in\mathcal{R}_{s,n}(K)}}
H(D')^{m-s-t}.
\end{align}

Summation by parts of the last sum leads by Theorem \ref{th:schmidt} to
{\small$$
\underset{H(D')\ll\tfrac{1}{\beta^{ds}}}{\sum_{D'\in\mathcal{R}_{s,n}(K)}}H(D')^{m-s-t}
\ll \frac{1}{\beta^{ds(m-s-t)}}\underset{H(D')\ll\tfrac{1}{\beta^{ds}}}{\sum_{D'\in\mathcal{R}_{s,n}(K)}}1+\int_{1}^{\frac{1}{\beta^{ds}}}\Big( \underset{H(D')\ll u}{\sum_{D'\in\mathcal{R}_{s,n}(K)}}1\Big) u^{m-s-t-1}du$$
$$
\ll\frac{1}{\beta^{ds(n+m-s-t)}}+\int_{1}^{\frac{1}{\beta^{ds}}}u^{n+m-s-t-1}du.
$$
}

The exponent in the integral now satisfies the upper and lower bounds $n-t\le n+m-s-t-1<m-s-1$. In particular, if $n+m=s+t$, the last right-hand side is $\ll\log(1/\beta) \ll 1/\beta^{\varepsilon}$.
Otherwise, we get the upper bound $\ll\tfrac{1}{\beta^{ds(n+m-s-t)}}$. In both cases, there is an upper bound $\ll \tfrac{\log(1/\beta)}{\beta^{ds(n+m-s-t)}}$.

In both cases, plugging in our bound for the sum over $\mathcal{R}_{s,n}^{\beta}(K)$, 
we obtain that 

$$\sum_{D'\in\mathcal{R}_{s,n}^{\beta}(K)}\sum_{v\in M_t(\Lambda_{D'})}|g(\beta v)|\beta^{mtd}n^{\beta}_{m,n}(D',K)\ll\beta^{d \cdot  N }\log(1/\beta),$$
where 
\begin{align}
	N  & =  t(m-s) -(n-s)(m-s) - s(n+m-s-t) \\
	   & = (t-n)(m-s) +s(t-n) = m(t-n) \geq m.
\end{align}

In conclusion, we therefore in particular obtain as $\beta\rightarrow0$ the estimate
$$	\sum_{D\in\mathcal{R}_{m,n}^{\beta}(K)}\underset{\rank v<m}{\sum_{v\in M_t(\Lambda_D)}}|g(\beta v)|\beta^{mtd}\ll\beta^{2d}\log(1/\beta).$$

 \section{Bounds for the shortest vector}
\label{se:bounds_for_vectors}
 
We turn to some applications. Theorem \ref{th:higher_moments} shows that moments obtained for lifts of codes to $\OK$-modules converge to the Rogers integral formula for the space of free $\OK$-module lattices $\SL_{t}(\KR)/\SL_{t}(\OK)$ studied in \cite{GSV23} and therefore to the moments of free module lattices. Moreover, the same moment formulas hold for the whole space of module lattices $\ML_{t}(K)$. The work of evaluating the Rogers integral formula and computing the moments that was carried through in \cite{GSV23} therefore remains valid in both the setting of Theorem \ref{th:higher_moments} and of module lattices $\ML_{t}(K)$. \par
We shall highlight results valid for number fields of increasing degree. We state them for module lattices, noting that the same results hold for the discrete sets of lifts of codes coming via reductions modulo suitable primes of large enough norm as in Theorem \ref{th:higher_moments}.
\begin{theorem}\label{thm:mainGSV}
		Let $\mathcal{S}$ denote any set of number fields $K$ such that the absolute Weil height of elements in $K^\times\setminus\mu(K)$ has a strictly positive uniform lower bound on $\mathcal{S}$. There are then for a given $n$ explicit constants $t_0(n,\mathcal{S})=O_\mathcal{S}(n^3\log\log n)$ as well as explicit constants $C,\varepsilon>0$, all uniform in $\mathcal{S}$, 
		such that for any $t> t_0$ and for any $K\in\mathcal{S}$ of degree $d$ the $n$-th moment of the number of nonzero points on $\calO_K$-module lattices $\Lambda\in \ML_{t}(K_i)$ inside an origin-centered ball $B$ of volume $V$ satisfies:
		\begin{align}
			\label{eq:small_error_from_poisson}
			\omega_K^n  \cdot m_n( \tfrac{V}{\omega_{ K} } )
	       &\leq \mathbb{E}[\left( \card B \cap \Lambda \setminus \{0\}\right)^n]
	       \le 
	       \omega_K^n\cdot  m_n( \tfrac{V}{\omega_{ K }} )
	       + E_{n,t,K} \cdot (V+1)^{n-1}
		\end{align}
		with error term $E_{n,t,K}$ satisfying
		\begin{align}
			E_{n,t,K} \le 
	       C \cdot (td)^{(n-2)/2}\cdot \omega_K^{n^2/4}\cdot Z(K,t,n)\cdot e^{-\varepsilon \cdot d(t-t_0)} .
		\end{align}
		Here $\omega_{K}=\# \mu(K) $ are the number of roots of unity in $K$, $ Z(K,t,n)$ denotes a finite product of Dedekind zeta values $\zeta_{K}$ at certain real values $>1$ and 
		$m_n(\lambda)=e^{-\lambda}\sum_{r=0}^\infty\frac{\lambda^{r}}{r!} r^{n}  $ denotes the $n$-th moment of a Poisson distribution of parameter $\lambda$.
	\end{theorem}
For a proof and explicit expressions for the constants and zeta factors we refer the reader to \cite[Section 5]{GSV23}. The absolute Weil height of an algebraic number $\alpha\in\overline{\mathbb{Q}}^\times$ is given by the sum over the places $M_\alpha$ of $\mathbb{Q}(\alpha)$:
$$h(\alpha):=\tfrac{1}{\deg(\alpha)}\cdot \sum_{v\in M_\alpha} \log \max\{1, |\alpha |_v\}$$
and is a measure of its arithmetic complexity. In \cite{GSV23}, it is a key input helping to control error terms when evaluating the Rogers integral formula for $\OK$-modules. Infinite extensions of $\mathbb{Q}$ satisfying the absolute Weil height lower bound of Theorem \ref{thm:mainGSV} are said to possess the \emph{Bogomolov property}. A typical example of an infinite tower of number fields with the Bogomolov property are the cyclotomic numbers $\mathbb{Q}^{cyc}=\bigcup_{i\geq 2} \mathbb{Q}(\zeta_i)$. In particular, we obtain the following second moment result with explicit constants: 
\begin{proposition}\label{cor:introcyclo}
	Consider a sequence of cyclotomic number fields given by $K_i=\mathbb{Q}(\zeta_{k_i})$ of degree $d_i=\varphi(k_i)$ and let $t_0=\tfrac{267}{10}$. Let $\Lambda \in \ML_{t}(K_i)$ be a Haar-random lattice and let $B$ be an origin-centered ball of volume $V$ with respect to the measure from (\ref{eq:norm}).
There then exist uniform  constants $C,\varepsilon> 0$ such that for any $t\geq 27$ and for any degree $d_i$, we have 
$$ V^2+V\cdot {k_i}\leq\mathbb{E}[\left( \card B \cap \Lambda \setminus \{0\}\right)^2]\leq V^2+V\cdot {k_i}\cdot (1+C\cdot e^{-\varepsilon\cdot d_i(t-t_0)}).$$
\end{proposition}
Similar results can be derived for other sequences of number fields known to satisfy the Bogomolov property, such as totally real fields or infinite Galois extensions whose Galois group has finite exponent modulo its center. With this in hand, we may derive Theorem \ref{th:as_module_bound}.

 \begin{proof} 
 {\bf (of Theorem \ref{th:as_module_bound})}

 Let $K = \mathbb{Q}(\zeta_{k})$ be our number field. Let $\varepsilon(k)\rightarrow 0$ be a function of $k$.
 Let $B$ be an origin-centered ball of volume $V$ and let $\Lambda \in \ML_{t}(K)$ be a Haar-random lattice.
 Denote $\rho(\Lambda) = \card ( \Lambda \cap B \setminus \{0\} )$.
 Note that because of (generalizations of) Siegel's mean value theorem, we have that 
 $\mathbb{E}(\rho(\Lambda))  = V$. We know that $\rho(\Lambda)\in k \cdot \mathbb{Z}_{\geq 0}$ because of $\zeta_k$ acting on lattice vectors. 

 Let $m \in \mathbb{Z}_{\geq 1}$ be arbitrary. For $x \in \mathbb{Z}_{\geq 0}$, observe that the following inequality holds concerning the indicator function of $x=0$:
 \begin{equation}
1- x\leq \mathbf{1}( x=0 ) \leq  {m}^{-2}(x-m)^{2}.
\label{eq:sandwich}
 \end{equation}
 Let us set $x= \rho(\Lambda)/k$ and evaluate the expected value.
 Using the left side of \ref{eq:sandwich}
 we get that when $V = k \cdot \varepsilon(k)$, with a probability of at least $1-\varepsilon(k)$, $\rho(\Lambda) = 0$ and therefore
 \begin{equation}
  k^{-\frac{1}{t \varphi(k)}}\frac{ \lambda_1(\Lambda) }{ \gamma( t  \varphi(k))} \geq  ( \varepsilon(k))^{\frac{1}{t \varphi(k)}} \simeq 1 + \frac{\log (\varepsilon(k))}{ t \varphi(k)}.
 \end{equation}
 On the other hand, suppose we set $V=km$ as a function of the free parameter $m$.
 Then, for some $\varepsilon_1 > 0$, 
 we get from the right side of \ref{eq:sandwich} and Proposition \ref{cor:introcyclo} that 
 \begin{align}
	 \mathbb{P}(\rho(\Lambda) = 0)  & \leq  \frac{1}{m^{2}}\left(\frac{V^{2}}{k^{2}} + \frac{V}{k}\right) - 2\frac{V}{mk} +1 + O(m^{-2}e^{- \varepsilon_1 \varphi(k)}) \\
					& =  \frac{1}{m} +  O(m^{-2}e^{- \varepsilon_1 \varphi(k)}) .
 \end{align}
 Set $m = \varepsilon(k)^{-1}$. It follows that with
 probability at least $1-\varepsilon(k) + o(\varepsilon(k)^2)$ we obtain $\rho(\Lambda) \neq 0$. This means that 
 \begin{equation}
   k^{-\frac{1}{t\varphi(k)}}\frac{\lambda_1(\Lambda)}{ \gamma(t \varphi(k))} 
   \leq \varepsilon(k)^{ - \frac{1}{t \varphi(k)}} \simeq 1- \frac{\log \varepsilon(k) }{t \varphi(k)}.
 \end{equation}
 The result has been presented for $\varepsilon(k) = \frac{1}{\log(k)}$.
 \end{proof}

\bibliographystyle{unsrt}
\bibliography{authfile}

\end{document}